\documentclass[12pt,english,reqno]{amsart}
\usepackage[T1]{fontenc}
\usepackage[utf8]{inputenc}
\usepackage[a4paper]{geometry}
\usepackage{amstext}
\usepackage{amsthm}
\usepackage{amssymb}
\usepackage{setspace}
\makeatletter
\numberwithin{equation}{section}
\numberwithin{figure}{section}
\theoremstyle{plain}
\newtheorem{thm}{\protect\theoremname}[section]
\theoremstyle{definition}
\newtheorem{defn}[thm]{\protect\definitionname}
\theoremstyle{remark}
\newtheorem{observation}[thm]{\protect\observationname}
\theoremstyle{plain}
\newtheorem{cor}[thm]{\protect\corollaryname}
\theoremstyle{remark}
\newtheorem{rem}[thm]{\protect\remarkname}
\theoremstyle{plain}
\newtheorem{lem}[thm]{\protect\lemmaname}
\theoremstyle{remark}
\newtheorem{claim}[thm]{\protect\claimname}

\usepackage{amsmath}
\usepackage{bbm}
\usepackage{dsfont}

\usepackage[unicode=true,pdfusetitle,
 bookmarks=true,bookmarksnumbered=false,bookmarksopen=false,
 breaklinks=false,pdfborder={0 0 0},pdfborderstyle={},backref=false,colorlinks=false]
 {hyperref}

\usepackage{tikz}
\usetikzlibrary{arrows,decorations.markings,shapes}
\usetikzlibrary{patterns}

\usepackage{etoolbox}
\patchcmd{\@settitle}{\uppercasenonmath\@title}{}{}{}
\patchcmd{\@setauthors}{\MakeUppercase}{}{}{}
\patchcmd{\section}{\scshape}{}{}{}

\makeatother

\usepackage{babel}
\providecommand{\claimname}{Claim}
\providecommand{\corollaryname}{Corollary}
\providecommand{\definitionname}{Definition}
\providecommand{\lemmaname}{Lemma}
\providecommand{\observationname}{Observation}
\providecommand{\remarkname}{Remark}
\providecommand{\theoremname}{Theorem}

\begin{document}
\global\long\def\One{\mathds{1}}%

\global\long\def\Laplacian{\Delta}%

\global\long\def\grad{\operatorname{grad}}%

\global\long\def\div{\operatorname{div}}%

\global\long\def\norm#1{\left\Vert #1\right\Vert }%

\global\long\def\Z{\mathbb{Z}}%

\global\long\def\R{\mathbb{R}}%

\global\long\def\C{\mathbb{C}}%

\global\long\def\N{\mathbb{N}}%

\global\long\def\T{\mathbb{T}}%

\global\long\def\P{\mathbb{P}}%

\global\long\def\E{\mathbb{E}}%

\global\long\def\cL{\mathcal{L}}%

\global\long\def\cD{\mathcal{D}}%

\global\long\def\cH{\mathcal{H}}%

\global\long\def\floor#1{\left\lfloor #1\right\rfloor }%

\global\long\def\ceil#1{\left\lceil #1\right\rceil }%

\global\long\def\var{\operatorname{Var}}%

\global\long\def\cov{\operatorname{Cov}}%

\global\long\def\dd{\operatorname{d}}%

\title{A Thomson-type variational principle for diffusion coefficients}
\author{Assaf Shapira}
\address{Université Paris Cité, CNRS, MAP5, F-75006 Paris, France }
\email{assaf.shapira@math.cnrs.fr}
\urladdr{https://assafshap.github.io}
\begin{abstract}
We consider reversible interacting particle systems with conserved
number of particles. A standard variational formulation describes
the diffusion coefficient of such models as the infimum of a certain
functional. The purpose of this paper is to derive a new, alternative,
variational characterization, as the \emph{supremum} of another functional.
This is a more natural framework when one is interested in obtaining
lower bounds on the diffusion coefficient. We present a specific example
of a kinetically constrained lattice gas where this variational principle
can be applied.
\end{abstract}

\maketitle

\section{Introduction}

The diffusion coefficient is an important characteristic of reversible
interacting particle systems with a conserved quantity. It is given
by a positive semidefinite matrix $D$, depending on the density $\rho$.
Generically \cite{KipnisLandim,Spohn2012IPS}, one expects a density
profile converging at diffusive scale to a solution of the diffusion
equation:
\[
\partial_{t}\rho(t,x)=\nabla(D(\rho(t,x))\nabla\rho(t,x)).
\]
It is, however, instructive to consider the diffusion coefficient
even without directly relating it to a hydrodynamic limit; see, e.g.,
\cite{Shapira23KA_HL,Shapira2024Noncooperative,AritaKrapivskyMallik17Diffusion_coefficient,AritaKrapivskyMallick2018,SchulerMessinaNastar20KineCluE,Sasada18GreenKubo}
for both theoretical and numerical studies. Understanding the behavior
of $D$ gives a good candidate both for a hydrodynamic limit and for
the equilibrium fluctuation field, and provides insights on time scales
of the model. In particular, the \emph{positivity} of the diffusion
coefficient implies (at least in a soft sense \cite{Shapira23KA_HL})
that the density profile evolves in diffusive time scales; while $0$
diffusion coefficient means that the evolution is much slower than
diffusive.

One may see $D$, via a fluctuation-dissipation relation, as describing
the space-time density correlations in equilibrium, see \cite[Section II.2.2]{Spohn2012IPS}.
This is expressed in the Green-Kubo formula given in equation (\ref{eq:green-kubo}).
As observed in \cite[Section II.2.2]{Spohn2012IPS}, it is then possible
to find $D$ by solving the Poisson problem
\begin{equation}
\cL f=j_{l},\label{eq:poisson_problem}
\end{equation}
where $\cL$ is the infinitesimal generator of the process and $j_{l}$
is given in Definition \ref{def:phil_jl}. Proposition 2.2 of \cite[Part II]{Spohn2012IPS}
approaches this problem using variational methods, expressing the
solution (formally) as the minimizer of some Dirichlet energy. This
allows us to write $D$ as an infimum of a certain functional, acting
on local functions defined on the state space (see equation (\ref{eq:dirichlet})).
Thus, applying this functional to any test function of our choice,
we are able to bound the diffusion coefficient from above.

The purpose of this paper is to present an alternative variational
principle for the diffusion coefficient: rather than analyzing the
Poisson problem (\ref{eq:poisson_problem}) using Dirichlet's principle,
we consider Thomson's principle. This gives us a \emph{maximization}
problem over flow functions satisfying certain properties. Thus, $D$
is expressed as a \emph{supremum} of some functional, and can be bounded
from below by applying it to test flows. This is the content of the
main result, Theorem \ref{thm:thomson}. We remark that Thomson's
principle has been used to analyze several other observables of Markov
processes, for example hitting probabilities \cite{doyle_snell1984},
metastability \cite{bovierdenhollander2016metastability,LandimMarianiSeo2019dirichlet},
or homogenization \cite{Owhadi2003effective_conductivity}.

After proving this Thomson-type variational principle, we show how
to apply it in practice. Since test flows must satisfy certain conditions,
finding them is not a trivial task. We will see three possible strategies
to construct such flows for one particular model, introduced by Bertini
and Toninelli in \cite{BertiniToninelli}. This model is a simple
example of a \emph{kinetically constrained lattice gas}, a family
of interacting particle systems devised in the physics literature
in order to study the liquid-glass transition; it can be seen as the
conservative version of \emph{kinetically constrained spin models}
\cite{HartarskyToninelli2024kcmbook}. Kinetically constrained lattice
gases have degenerate jump rates, slowing down the dynamics and making
$D$ very small at high densities. In particular, while the Bertini-Toninelli
model is known to have a strictly positive diffusion coefficient (\cite{GoncalvesLandimToninelli,Shapira2024Noncooperative},
see also equation (\ref{eq:BT_bound0})), it is unclear under which
conditions this holds in more generality. Moreover, kinetically constrained
lattice gases are notoriously difficult to analyze (especially out
of equilibrium) since they are not attractive, making techniques such
as monotone coupling unavailable. Variational tools are hence essential
to understanding these models.

We will improve the lower bound on $D$ for the Bertini-Toninelli
model, and more importantly develop techniques that could open the
door to analyzing other models in a larger setting, where even positivity
of the diffusion coefficient is unknown.

\subsection{Perspective and further questions}

The variational principle in Theorem \ref{thm:thomson} offers a promising
approach to studying diffusion coefficients of different interacting
particle systems. One example where variational methods are particularly
useful, is the study of kinetically constrained lattice gases. The
methods demonstrated here, in particular in Section \ref{subsec:unbounded_path},
should apply to other kinetically constrained lattice gases, including
some where we are currently not even able to prove that $D$ is nonzero.

The maximization principle we introduced could also be used in rigorous
numerics: \cite{AritaKrapivskyMallik17Diffusion_coefficient,AritaKrapivskyMallick2018},
for example, use the minimization principle (\ref{eq:dirichlet})
to estimate diffusion coefficients, obtaining rigorous upper bounds
converging to the true value. The work presented here could be used
in order to add exact lower bounds, providing a precise control of
the error in such numerical schemes.

A minimization principle in the spirit of equation (\ref{eq:dirichlet})
is given in \cite{Faggionato08} for reversible processes in a random
translation invariant ergodic environment. One may naturally ask if
a Thomson-type variational principle can be derived in this setting.

Finally, a similar but not identical question is that of the self-diffusion
coefficient, giving the scaling limit of a tracer, see \cite[Section II.6.2]{Spohn2012IPS}.
The self diffusion coefficient can also be written in terms of a Dirichlet-type
principle \cite[Proposition II.6.1]{Spohn2012IPS}, derived in a similar
manner to equation (\ref{eq:dirichlet}). Hence, it will be interesting
to see if an analysis in the spirit of Section \ref{sec:flows_and_proof}
can apply.

\medskip

The paper proceeds as follows. We start with a list of notation (we
will stay close throughout to the definitions and notation of \cite{Spohn2012IPS}).
Then Section \ref{sec:defs_and_thm} will present important definitions
and results. Section \ref{sec:flows_and_proof} is dedicated to the
proof of the main result, Theorem \ref{thm:thomson}. Finally, in
Section \ref{sec:BT} we apply Theorem \ref{thm:thomson} in three
different ways to the Bertini-Toninelli model.

\section{Notation}

These are some of the notation that will be introduced in the paper,
listed here for the reader's convenience.
\begin{itemize}
\item The model is defined on $\Z^{d}$. We denote $x\cdot y$ the standard
scalar product for $x,y\in\Z^{d}$, and $\norm x=\sqrt{x\cdot x}$
.
\item The configuration space is $\Omega=\{0,1\}^{\Z^{d}}$. The measure
$\mu$ is given by a product of $\text{Ber}(p)$ for some fixed $p\in(0,1)$.
We denote $q=1-p$.
\item For $\eta\in\Omega,$ $\eta^{x,y}$ is the configuration obtained
by exchanging the occupation values at $x$ and $y$.
\item $\overline{\eta}=1-\eta$.
\item $\tau_{x}$ is the translation by $x$: for $\eta\in\Omega$, 
\begin{itemize}
\item $\tau_{x}\eta(y)=\eta(y-x)$;
\item $\tau_{x}(\eta^{yz})=(\tau_{x}\eta)^{y+x,z+x}$;
\item $\tau_{x}f(\eta)=f(\tau_{x}\eta)$.
\end{itemize}
\item $c(x,y;\eta)$ are positive transition rates for $x,y\in\Z^{d}$ and
$\eta\in\Omega.$ By convention $c(x,y;\eta)=c(y,x;\eta)$.
\item $\cL$ is the infinitesimal generator of the process with jump rates
$c(x,y;\eta)$.
\item $D$ is the diffusion coefficient. We study the product $l\cdot Dl$,
where $l$ denotes a generic vector in $\R^{d}$.
\item The inner product of two functions is given by 
\[
\left\langle f,g\right\rangle =\sum_{x\in\Z^{d}}\left(\mu\left(f^{*}\tau_{x}g)\right)-\mu(f^{*})\mu(g)\right).
\]
It is defined over the space $\cH_{1}$.
\item For $\phi,\psi$ admissible flows, the inner product is given by 
\[
\left\langle \phi,\psi\right\rangle =\frac{1}{4}\sum_{x,y,z\in\Z^{d}}\mu\left[c(y,z;\eta)^{-1}\phi(\eta,\eta^{y,z})^{*}\,\tau_{x}\psi(\eta,\eta^{y,z})\right].
\]
It is defined over the space $\cH_{2}$.
\item $V_{0}$ is the space of admissible flows $\phi$ such that $\left\langle \div\phi,\div\phi\right\rangle =0$.
\end{itemize}

\section{\label{sec:defs_and_thm}Definitions and main results}

We will consider exclusion processes, on the state space $\Omega=\{0,1\}^{\Z^{d}}$.
The evolution of the process is described by particle jumps with rates
given by a function $c:\Z^{d}\times\Z^{d}\times\Omega\to[0,\infty)$.
That is, for $\eta\in\Omega$ and $x,y\in\Z^{d}$, the occupation
values $\eta(x)$ and $\eta(y)$ are exchanged at rate $c(x,y;\eta)$.
We assume that the jump range is finite, i.e., $c(x,y;\eta)=0$ if
$\norm{x-y}$ is larger than some finite range, and translation invariant,
meaning that $c(x,y;\eta)=c(x+z,y+z;\tau_{z}\eta)$ for any $z\in\Z^{d}$,
where $\tau_{z}\eta$ is the translated configuration:
\[
(\tau_{z}\eta)(x)=\eta(x-z).
\]

The process can be constructed via its infinitesimal generator, acting
on a local function $f:\Omega\to\C$ as 
\begin{equation}
\cL f(\eta)=\frac{1}{2}\sum_{x,y\in\Z}c(x,y;\eta)\left(f(\eta^{x,y})-f(\eta)\right),
\end{equation}
where $\eta^{x,y}$ is the configuration obtained from $\eta$ after
exchanging $x,y$:
\[
\eta^{x,y}(z)=\begin{cases}
\eta(z) & \text{if }z\notin\{x,y\},\\
\eta(y) & \text{if }z=x,\\
\eta(x) & \text{if }z=y.
\end{cases}
\]
The factor $\frac{1}{2}$ comes form the double counting of the pair
$(x,y)$ and $(y,x)$; we assume by convention $c(x,y;\eta)=c(y,x;\eta)$.

We assume that $\cL$ is reversible with respect the measure $\mu$,
given by an IID product of $\text{Ber}(p)$ random variables, for
a fixed parameter $p\in(0,1)$; we also denote $q=1-p$. This is the
same as saying that $c(x,y;\eta)=c(x,y;\eta^{x,y})$ for all $\eta\in\Omega$
and $x,y\in\Z^{d}$.

Our object of interest is the diffusion coefficient, which is the
symmetric $d\times d$ matrix given by the Green-Kubo formula \cite[Proposition II.2.1]{Spohn2012IPS}
\begin{equation}
l\cdot Dl=\frac{1}{\chi}\left(\sum_{x}\mu\left[c(0,x;\eta)(\eta(0)-\eta(x))^{2}\right](l\cdot x)^{2}-\int_{0}^{\infty}dt\sum_{x}\mu\left[\tau_{x}j_{l}\ e^{t\cL}j_{l}\right]\right),\qquad l\in\R^{d},\label{eq:green-kubo}
\end{equation}
where $j_{l}$ is the current at the origin that we define later on
(Definition \ref{def:phil_jl}), and $\chi=pq$ is the compressibility.

In order to formulate the result, we need to consider flows on $\Omega$:
\begin{defn}
A \emph{flow} is an antisymmetric function $\phi:\Omega^{2}\to\C$,
i.e., $\phi(\eta,\eta')=-\phi(\eta',\eta)$ for all $\eta,\eta'\in\Omega$.
We say that a flow is \emph{admissible} if $\phi(\eta,\eta')=0$ unless
$\eta'=\eta^{x,y}$ for a pair $x,y\in\Z^{d}$ such that $c(x,y;\eta)\neq0$.
\end{defn}

Next, we define the gradient of a function and divergence of a flow.
\begin{defn}
For a function $f:\Omega\to\R$, we define the gradient to be the
admissible flow given by
\begin{equation}
\grad f(\eta,\eta^{x,y})=c(x,y;\eta)\ \left(f(\eta^{x,y})-f(\eta)\right),\qquad\eta\in\Omega,\ x,y\in\Z^{d}.
\end{equation}
In this formula, and the ones to follow, it is implicit that $\grad f(\eta,\eta')=0$
elsewhere (i.e., if there are no $x,y\in\Z^{d}$ such that $\eta'=\eta^{x,y}$).

For a flow $\phi:\Omega^{2}\to\R$, the divergence is a function on
$\Omega$ given by
\begin{equation}
\div\phi(\eta)=\sum_{\eta'}\phi(\eta,\eta'),\label{eq:def_div}
\end{equation}
when the sum is well-defined.
\end{defn}

A direct application of this definition yields:
\begin{observation}
Let $f:\Omega\to\R$ be a local function. Then $\div\grad f$ is well
defined, and given by $\cL f$.
\end{observation}

We now define the following function and flow, describing the particle
current at a given direction $l$:
\begin{defn}
\label{def:phil_jl}Let $l\in\R^{d}$. The flow of particle current
adjacent to $0$ is defined as
\begin{equation}
\phi_{l}(\eta,\eta^{0,z})=\frac{1}{2}(z\cdot l)\ c(0,z;\eta)\left(\eta(0)-\eta(z)\right),\qquad z\in\Z^{d},
\end{equation}
and the overall current adjacent to $0$ is the function 
\begin{equation}
j_{l}(\eta)=\frac{1}{2}\sum_{z\in\Z^{d}}(z\cdot l)\ c(0,z;\eta)\left(\eta(0)-\eta(z)\right).
\end{equation}
The factor $\frac{1}{2}$ is added due to the convention, that when
a particle jumps between $x$ and $y$ we associate half of the corresponding
current to $x$ and the other half to $y$.
\end{defn}

Our next definition consists of two (degenerate) inner products, on
functions and on flows, inspired by the second term of equation (\ref{eq:green-kubo}),
see also \cite[equation II.2.31]{Spohn2012IPS}.
\begin{defn}
For two local functions $f,g:\Omega\to\C$, define 
\begin{equation}
\left\langle f,g\right\rangle =\sum_{x\in\Z^{d}}\left(\mu(f^{*}\tau_{x}g)-\mu(f^{*})\mu(g)\right).
\end{equation}

For two local admissible flows $\phi,\psi$, define 
\begin{equation}
\left\langle \phi,\psi\right\rangle =\frac{1}{4}\sum_{x,y,z\in\Z^{d}}\mu\left[c(y,z;\eta)^{-1}\phi(\eta,\eta^{y,z})^{*}\,\tau_{x}\psi(\eta,\eta^{y,z})\right].
\end{equation}

Note that, thanks to the IID structure of $\mu$, these inner products
are indeed well defined on the spaces of local functions and local
admissible flows. In the rest of this section, just as in \cite[Proof of Proposition II.2.2 and equation II.2.155]{Spohn2012IPS},
we consider the Hilbert spaces $\cH_{1}$ and $\cH_{2}$ given by
(equivalence classes of) their completion.
\end{defn}

\begin{defn}
\label{def:V_0}$V_{0}$ is the space of admissible flows with vanishing
divergence, that we write formally as
\begin{equation}
V_{0}=\left\{ \phi\in\cH_{2}:\left\langle \div\phi,\div\phi\right\rangle =0\right\} .
\end{equation}
For $\phi\in\cH_{2}$, $\left\langle \div\phi,\div\phi\right\rangle =0$
should be understood by the completion under $\left\langle \cdot,\cdot\right\rangle $.
That is, there exists a sequence $(\phi_{n})_{n\in\N}$ of flows with
well defined divergence that converge in $\cH_{2}$ to $\phi$, such
that $\left\langle \div\phi_{n},\div\phi_{n}\right\rangle \xrightarrow{n\to\infty}0$.
\end{defn}

We are now ready to formulate the Thomson-type variational principle
for the diffusion coefficients.
\begin{thm}
\label{thm:thomson}Let $l\in\R^{d}$, and consider the diffusion
coefficient $D$ of the interacting particle system above. Then 
\begin{equation}
l\cdot Dl=\frac{1}{\chi}\sup_{\phi\in V_{0}}\left[2\left\langle \phi_{l},\phi\right\rangle -\left\langle \phi,\phi\right\rangle \right],\label{eq:thomson}
\end{equation}
where $\chi=pq$ is the compressibility.
\end{thm}

It is sometimes convenient to optimize, for a given $\phi\in V_{0}$,
over flows of the form $\lambda\phi$ for $\lambda\in\R$. A quick
calculation yields the following corollary:
\begin{cor}
\label{cor:optimize_flow}Let $\phi\in V_{0}$ with $\left\langle \phi,\phi\right\rangle \neq0$.
Then 
\begin{equation}
l\cdot Dl\ge\frac{1}{\chi}\ \frac{\left\langle \phi_{l},\phi\right\rangle ^{2}}{\left\langle \phi,\phi\right\rangle }.
\end{equation}
\end{cor}

\begin{rem}
We assume $\mu$ to be an IID product measure on $\{0,1\}^{\Z^{d}}$
for simplicity of the presentation; there is however no major obstacle
in proving the result for more general reversible measures and more
general state spaces.

One technicality that should be addressed is that in order for the
different objects we use to be well defined we must require strong
enough spatial mixing properties of $\mu$.

The second issue arises from the fact that if $c(x,y;\eta)\neq c(x,y;\eta^{x,y})$
then the gradient of a function is no longer antisymmetric. In order
to rectify this problem one must redefine a flow with ``tilted''
antisymmetry, written formally as $\mu(\eta)\ \phi(\eta,\eta')=-\mu(\eta')\ \phi(\eta',\eta)$.
Under this definition, $\grad f$ is indeed a flow by reversibility.
This modification will change the proof of Lemma \ref{lem:divgrad},
where we use the antisymmetry of $\phi$, but the additional factor
introduced will be canceled by the Radon-Nikodym derivative under
the change of variable $\eta\mapsto\eta^{yz}$.
\end{rem}

\begin{rem}
\label{rem:dirichlet}One may compare Theorem \ref{thm:thomson} with
the standard, ``Dirichlet-type'', variational principle for the
diffusion coefficient \cite[equations II.2.33 and II.2.34]{Spohn2012IPS},
which can be written (see equation (\ref{eq:phil_phil})) as 
\begin{multline}
l\cdot Dl=\frac{1}{\chi}\left[\left\langle \phi_{l},\phi_{l}\right\rangle +\inf_{f}\left\{ -\left\langle f,\cL f\right\rangle +2\left\langle f,j_{l}\right\rangle \right\} \right]\\
=\frac{1}{\chi}\inf_{f}\left\{ \frac{1}{4}\sum_{z\in\Z^{d}}\mu\left[c(0,z;\eta)\left((l\cdot z)(\eta(0)-\eta(z))+\sum_{x}\left(\tau_{x}f(\eta^{0,z})-\tau_{x}f(\eta)\right)\right)\right]^{2}\right\} ,\label{eq:dirichlet}
\end{multline}
with the infimum taken over all local functions $f$. We also note
here that, formally, the infimum is attained at $f$ solving the Poisson
problem (\ref{eq:poisson_problem}), so $\inf_{f}\left\{ -\left\langle f,\cL f\right\rangle +2\left\langle f,j_{l}\right\rangle \right\} =\left\langle j_{l},\cL^{-1}j_{l}\right\rangle $.
\end{rem}

\begin{rem}
\label{rem:gradient}In the case of a gradient model, the current
can be expressed as a gradient of a function. In particular, $\sum_{x}\tau_{x}\div\phi_{l}=\sum_{x}\tau_{x}j_{l}=0$,
hence $\phi_{l}\in V_{0}$. Then the lower bound obtained by taking
the test flow $\phi=\phi_{l}$ in equation (\ref{eq:thomson}) coincides
with the upper bound given by (\ref{eq:dirichlet}) with the test
function $f\equiv0$. It is shown in \cite{Sasada18GreenKubo} that
this condition is also necessary, i.e., if $\phi_{l}\in V_{0}$ for
all \textbf{$l$}, then $l\cdot Dl=\frac{1}{\chi}\left\langle \phi_{l},\phi_{l}\right\rangle $
for all $l$, and the model must be gradient.
\end{rem}

The last part of the paper is dedicated to the application of Theorem
\ref{thm:thomson} to a specific model, introduced by Bertini and
Toninelli in \cite{BertiniToninelli}. We discuss the one dimensional
setting, where the diffusion coefficient is a non-negative real number.
While the main purpose of this part is to demonstrate how to construct
flows to be used in equation (\ref{eq:thomson}), the result also
improves on previously known bounds, illuminating qualitative properties
of the model (see Remark \ref{rem:BT_qualitative}).
\begin{thm}
\label{thm:BT}Let $D$ be the diffusion coefficient of the Bertini-Toninelli
model, defined via the transition rates given in equation (\ref{eq:c_BT}).
Then the following lower bounds hold:
\begin{eqnarray*}
D & \ge & \frac{2q}{1+q},\\
D & \ge & \frac{9q^{2}}{7-4q+6q^{2}},\\
D & \ge & \frac{1}{5000}pq^{9}.
\end{eqnarray*}
In particular $D\neq0$.
\end{thm}

\begin{rem}
The first bound is the best of the three for all values of $q$; the
main reason we give the two others is that their proofs involve techniques
that could be useful to other models. At the same time, the three
test flows we construct can all be used together to obtain a quantitive
bound strictly better than $\frac{2q}{1+q}$. This is explained in
Remark \ref{rem:combined_flow}, but since computations become quickly
cumbersome we will only see an explicit expression at $q=1/2$, improving
$D\ge\frac{2}{3}$ of Theorem \ref{thm:BT} with the (slightly better)
$D\ge0.691$.
\end{rem}

\section{\label{sec:flows_and_proof}Flows and proof of Theorem \ref{thm:thomson}}
\begin{lem}
For local $f:\Omega\to\C$,
\begin{equation}
\left\langle f,f\right\rangle =\lim_{N\to\infty}\frac{1}{\left|\Lambda_{N}\right|}\mu\left(\left|\sum_{x\in\Lambda_{N}}\left(\tau_{x}f-\mu(f)\right)\right|^{2}\right),
\end{equation}
where $\Lambda_{N}=\left([-N,N]\cap\Z\right)^{d}$.

For a local admissible flow $\phi$,
\begin{equation}
\left\langle \phi,\phi\right\rangle =\frac{1}{4}\sum_{z\in\Z^{d}}\mu\left[c(0,z;\eta)^{-1}\left|\sum_{x\in\Z^{d}}\tau_{x}\phi(\eta,\eta^{0,z})\right|^{2}\right].\label{eq:phi_phi}
\end{equation}
By completion both formulas hold for any $f\in\cH_{1}$ and $\phi\in\cH_{2}$.
\end{lem}

\begin{proof}
For the first equality, let $f$ be a local function, depending only
on the configuration in $\Lambda_{n}$ for some $n$, and assume without
loss of generality $\mu(f)=0$. We assume that $\mu$ is a product
measure, so $f$ is independent of $\tau_{x}f$ for $x\notin\Lambda_{n}$.
Then for large enough $N$,
\begin{eqnarray}
\left\langle f,f\right\rangle  & = & \sum_{x\in\Z^{d}}\mu(f^{*}\tau_{x}f)=\frac{1}{\left|\Lambda_{N}\right|}\sum_{y\in\Lambda_{N}}\sum_{x\in\Z^{d}}\mu(\tau_{y}f^{*}\ \tau_{x+y}f)=\frac{1}{\left|\Lambda_{N}\right|}\sum_{y\in\Lambda_{N}}\sum_{x\in\Lambda_{N+n}}\mu(\tau_{y}f^{*}\ \tau_{x}f)\\
 & = & \frac{1}{\left|\Lambda_{N}\right|}\mu\left(\left|\sum_{x}\tau_{x}f\right|^{2}\right)+\frac{1}{\left|\Lambda_{N}\right|}\sum_{y\in\Lambda_{N}}\sum_{x\in\Lambda_{N+n}\setminus\Lambda_{n}}\mu(\tau_{y}f^{*}\ \tau_{x}f).
\end{eqnarray}
In the second term, we can replace $\sum_{y\in\Lambda_{N}}\sum_{x\in\Lambda_{N+n}\setminus\Lambda_{n}}$
with $\sum_{x\in\Lambda_{N}\setminus\Lambda_{n}}\sum_{y\in x+\Lambda_{n}}$,
bounding it by $C_{f}\ \frac{\left|\Lambda_{N+n}\setminus\Lambda_{n}\right|}{\left|\Lambda_{N}\right|}$
which goes to $0$.

For the second part, let $\phi$ be a local admissible flow, depending
on the configuration in a box $\Lambda_{n}$. Then (using translation
invariance of $\mu$)
\begin{eqnarray*}
\left\langle \phi,\phi\right\rangle  & = & \frac{1}{4}\sum_{x,y,z\in\Z^{d}}\mu\left[c(y,z;\eta)^{-1}\phi(\eta,\eta^{y,z})^{*}\,\tau_{x}\phi(\eta,\eta^{y,z})\right]\\
 & = & \frac{1}{4}\sum_{x,y,z'\in\Z^{d}}\mu\left[c(0,z';\eta)^{-1}\phi(\tau_{y}\eta,\tau_{y}(\eta{}^{0,z'}))^{*}\,\tau_{x}\phi(\tau_{y}\eta,\tau_{y}(\eta{}^{0,z'}))\right]\\
 & = & \frac{1}{4}\sum_{x',y,z'\in\Z^{d}}\mu\left[c(0,z';\eta)^{-1}\tau_{y}\phi(\eta,\eta{}^{0,z'})^{*}\,\tau_{x'}\phi(\eta,\eta{}^{0,z'})\right],
\end{eqnarray*}
which concludes the proof. 
\end{proof}
\begin{claim}
Recall Definition \ref{def:phil_jl}. Then
\begin{equation}
j_{l}=\div\phi_{l},\label{eq:j_div_phi}
\end{equation}
\begin{equation}
\sum_{x}\tau_{x}\phi_{l}\ (\eta,\eta^{y,z})=((z-y)\cdot l)\ c(y,z;\eta)\left(\eta(y)-\eta(z)\right),\text{ and}
\end{equation}
\begin{equation}
\left\langle \phi_{l},\phi_{l}\right\rangle =\frac{1}{4}\sum_{z\in\Z^{d}}(z\cdot l)^{2}\mu\left[c(0,z;\eta)\left(\eta(0)-\eta(z)\right)^{2}\right].\label{eq:phil_phil}
\end{equation}
\end{claim}

\begin{proof}
The first equation is just rewriting of the definition of $j_{l}$.
For the second equation,
\begin{eqnarray*}
\sum_{x}\tau_{x}\phi_{l}(\eta,\eta^{y,z}) & = & \sum_{x}\phi_{l}(\tau_{x}\eta,(\tau_{x}\eta)^{y+x,z+x})\\
 & = & \phi_{l}(\tau_{-y}\eta,(\tau_{-y}\eta)^{0,z-y})+\phi_{l}(\tau_{-z}\eta,(\tau_{-z}\eta)^{0,y-z})\\
 & = & \frac{1}{2}((z-y)\cdot l)\ c(0,z-y;\tau_{-y}\eta)\left(\tau_{-y}\eta(0)-\tau_{-y}\eta(z-y)\right)\\
 &  & \quad+\frac{1}{2}((y-z)\cdot l)\ c(0,y-z;\eta)\left(\tau_{-z}\eta(0)-\tau_{-z}\eta(y-z)\right)\\
 & = & \frac{1}{2}((z-y)\cdot l)\ c(y,z;\eta)\left(\eta(y)-\eta(z)\right)+\frac{1}{2}((y-z)\cdot l)\ c(z,y;\eta)\left(\eta(z)-\eta(y)\right)\\
 & = & ((z-y)\cdot l)\ c(y,z;\eta)\left(\eta(y)-\eta(z)\right).
\end{eqnarray*}
Putting this identity in equation (\ref{eq:phi_phi}) yields equation
(\ref{eq:phil_phil}).
\end{proof}
\begin{lem}
\label{lem:divgrad}The divergence and gradient are adjoint up to
a sign: $\div=-\grad^{*}$. That is, for local $f:\Omega\to\C$ and
a local admissible flow $\phi$,
\begin{equation}
\left\langle \phi,\grad f\right\rangle =-\left\langle \div\phi,f\right\rangle .
\end{equation}
In particular, $\cL=\div\grad$ is negative self-adjoint.
\end{lem}

\begin{proof}
Without loss of generality $\mu(f)=0$. Using the antisymmetry of
$\phi$,
\begin{eqnarray*}
\left\langle \phi,\grad f\right\rangle  & = & \frac{1}{4}\sum_{x,y,z\in\Z^{d}}\mu\left[c(y,z;\eta)^{-1}\tau_{x}\phi(\eta,\eta^{y,z})^{*}\ c(y,z;\eta)\ \left(f(\eta^{y,z})-f(\eta)\right)\right]\\
 & = & \frac{1}{4}\sum_{x,y,z\in\Z^{d}}\mu\left[\tau_{x}\phi(\eta,\eta^{y,z})^{*}\left(f(\eta^{y,z})-f(\eta)\right)\right]\\
 & = & -\frac{1}{4}\sum_{x,y,z\in\Z^{d}}\mu\left[\tau_{x}\phi(\eta^{y,z},\eta)^{*}f(\eta^{y,z})\right]-\frac{1}{4}\sum_{x,y,z\in\Z^{d}}\mu\left[\tau_{x}\phi(\eta,\eta^{y,z})^{*}f(\eta)\right]\\
 & = & -\frac{1}{2}\sum_{x,y,z\in\Z^{d}}\mu\left[\tau_{x}\phi(\eta,\eta^{y,z})^{*}f(\eta)\right].
\end{eqnarray*}
It remains only to note that $\div\phi(\eta)=\frac{1}{2}\sum_{y,z\in\Z}\phi(\eta,\eta^{y,z})$
(since any $\eta'$ in the equation (\ref{eq:def_div}) is given by
$\eta'=\eta^{y,z}=\eta^{z,y}$).
\end{proof}

\begin{lem}
Let $V_{l}$ be the space of admissible flows $\phi\in\cH_{2}$ such
that $\left\langle \div\phi-j_{l},\div\phi-j_{l}\right\rangle =0$.
Then 
\begin{equation}
\left\langle j_{l},\cL^{-1}j_{l}\right\rangle =\sup_{\phi\in V_{l}}\left\{ -\left\langle \phi,\phi\right\rangle \right\} .
\end{equation}
The (formal) expression $\left\langle j_{l},\cL^{-1}j_{l}\right\rangle $
is defined as the term $-\int_{0}^{\infty}\dd t\sum_{x}\mu\left[\tau_{x}j_{l}\ e^{t\cL}j_{l}\right]$
appearing in the Green-Kubo formula (\ref{eq:green-kubo}).
\end{lem}

\begin{proof}
We start with the heuristics behind the proof: define formally $\phi_{*}=\grad f$
for $f=\cL^{-1}j_{l}$. First, note that $\div\phi_{*}=\div\grad f=j_{l}$,
hence $\phi_{*}\in V_{l}$. Then, on one hand,
\[
\left\langle \phi_{*},\phi_{*}\right\rangle =-\left\langle f,\div\grad f\right\rangle =-\left\langle f,\cL f\right\rangle =-\left\langle j_{l},\cL^{-1}j_{l}\right\rangle .
\]
On the other hand, any $\phi\in V_{l}$ can be written as $\phi_{*}+\delta$
for $\delta\in V_{0}$, hence 
\[
\left\langle \phi,\phi\right\rangle =\left\langle \phi_{*},\phi_{*}\right\rangle +2\left\langle \phi_{*},\delta\right\rangle +\left\langle \delta,\delta\right\rangle \ge\left\langle \phi_{*},\phi_{*}\right\rangle -2\left\langle f,\div\delta\right\rangle +0=\left\langle \phi_{*},\phi_{*}\right\rangle .
\]

Unfortunately, while the combination $\left\langle j_{l},\cL^{-1}j_{l}\right\rangle $
is finite (\cite[equation II.2.26]{Spohn2012IPS}), $\cL^{-1}j_{l}$
is not necessarily a well defined element of $\cH_{1}$. Hence, we
prove the lemma by approximated solutions of the Poisson problem (\ref{eq:poisson_problem}).
Define, for $T>0$, 
\begin{equation}
f_{T}=-\int\limits _{0}^{T}\dd t\ e^{t\cL}j_{l}.
\end{equation}
By \cite[equation II.2.26]{Spohn2012IPS}, $\left\langle j_{l},f_{T}\right\rangle $
converges as $T\to\infty$, to a limit that we denote $\left\langle j_{l},\cL^{-1}j_{l}\right\rangle $.
The notation $\cL^{-1}j_{l}$ comes from the fact that we think of
$f_{T}$ as an approximate solution of the Poisson problem (\ref{eq:poisson_problem}).
Indeed,
\[
\cL f_{T}=-\int\limits _{0}^{T}\dd t\ \frac{\dd}{\dd t}\left(e^{t\cL}j_{l}\right)=j_{l}-e^{T\cL}j_{l},
\]
with remainder satisfying
\[
\left\langle e^{T\cL}j_{l},e^{T\cL}j_{l}\right\rangle =\left\langle j_{l},e^{2T\cL}j_{l}\right\rangle \xrightarrow{T\to\infty}0
\]
since $\int^{\infty}\dd t\left\langle j_{l},e^{t\cL}j_{l}\right\rangle $
is positive and convergent. Moreover,
\[
\left\langle \cL f_{T},f_{T}\right\rangle =\left\langle j_{l},f_{T}\right\rangle -\left\langle e^{T\cL}j_{l},f_{T}\right\rangle =\left\langle j_{l},f_{T}\right\rangle -\left\langle j_{l},-\int\limits _{0}^{T}\dd t\ e^{(t+T)\cL}j_{l}\right\rangle \xrightarrow{T\to\infty}\left\langle j_{l},\cL^{-1}j_{l}\right\rangle .
\]

With this in mind, The proof of the lemma consists of two parts:
\begin{claim}
Let $(\phi_{n})_{n\in\N}$ be a sequence of admissible flows, with
$\left\langle \phi_{n},\phi_{n}\right\rangle <\infty$, satisfying
\begin{enumerate}
\item the Cauchy criterion $\left\langle \phi_{n}-\phi_{m},\phi_{n}-\phi_{m}\right\rangle \xrightarrow{n>m\to\infty}0$, 
\item \textbf{$\left\langle \div\phi_{n}-j_{l},g\right\rangle \xrightarrow{n\to\infty}0$}
for all $g:\Omega\to\C$ with $\left\langle g,g\right\rangle <\infty$.
\end{enumerate}
Then $\lim_{n\to\infty}\left\{ -\left\langle \phi_{n},\phi_{n}\right\rangle \right\} \le\left\langle j_{l},\cL^{-1}j_{l}\right\rangle $.
\end{claim}

\begin{proof}
Let $\phi_{T}=\grad f_{T}$. For any $T>0$,
\begin{eqnarray*}
\left\langle \phi_{n},\phi_{n}\right\rangle  & = & \left\langle \phi_{T},\phi_{T}\right\rangle +2\left\langle \phi_{T},\phi_{n}-\phi_{T}\right\rangle +\left\langle \phi_{n}-\phi_{T},\phi_{n}-\phi_{T}\right\rangle \\
 & \ge & -\left\langle f_{T},\div\phi_{T}\right\rangle -2\left\langle f_{T},\div\phi_{n}\right\rangle +2\left\langle f_{T},\div\phi_{T}\right\rangle \\
 & = & \left\langle f_{T},\cL f_{T}\right\rangle -2\left\langle f_{T},\div\phi_{n}-j_{l}\right\rangle -2\left\langle f_{T},j_{l}\right\rangle \\
 &  & \xrightarrow{n\to\infty}\left\langle f_{T},\cL f_{T}\right\rangle -2\left\langle f_{T},j_{l}\right\rangle .
\end{eqnarray*}
This is true for all $T$, hence 
\[
\lim_{n\to\infty}\left\langle \phi_{n},\phi_{n}\right\rangle \ge\limsup_{T>0}\left(\left\langle f_{T},\cL f_{T}\right\rangle -2\left\langle f_{T},j_{l}\right\rangle \right)=-\left\langle j_{l},\cL^{-1}j_{l}\right\rangle .
\]
\end{proof}
\begin{claim}
Let $\phi_{T}=\grad f_{T}$. Then:
\begin{enumerate}
\item $\phi_{T}$ is Cauchy, i.e., $\left\langle \phi_{T}-\phi_{T'},\phi_{T}-\phi_{T'}\right\rangle \xrightarrow{T>T'\to\infty}0$,
\item $\left\langle \div\phi_{T}-j_{l},\div\phi_{T}-j_{l}\right\rangle \xrightarrow{T\to\infty}0$,
\item $-\left\langle \phi_{T},\phi_{T}\right\rangle \xrightarrow{T\to\infty}\left\langle j_{l},\cL^{-1}j_{l}\right\rangle $.
\end{enumerate}
\end{claim}

\begin{proof}
First, we show that the sequence is Cauchy: for $T>T'$,
\begin{eqnarray*}
\left\langle \phi_{T}-\phi_{T'},\phi_{T}-\phi_{T'}\right\rangle  & = & -\left\langle f_{T}-f_{T'},\cL f_{T}-\cL f_{T'}\right\rangle \\
 & = & -\left\langle \int\limits _{T'}^{T}\dd t\ e^{t\cL}j_{l},e^{T\cL}j_{l}\right\rangle +\left\langle \int\limits _{T'}^{T}\dd t\ e^{t\cL}j_{l},e^{T'\cL}j_{l}\right\rangle \\
 &  & \xrightarrow{T>T'\to\infty}0.
\end{eqnarray*}
Next, note that 
\[
\left\langle \div\phi_{T}-j_{l},\div\phi_{T}-j_{l}\right\rangle =\left\langle \cL f_{T}-j_{l},\cL f_{T}-j_{l}\right\rangle \xrightarrow{T\to\infty}0.
\]
Finally,
\[
-\left\langle \phi_{T},\phi_{T}\right\rangle =\left\langle \cL f_{T},f_{T}\right\rangle \xrightarrow{T\to\infty}\left\langle j_{l},\cL^{-1}j_{l}\right\rangle .
\]
\end{proof}
These two claims prove our lemma. Indeed, the first shows that any
element of $V_{l}$ satisfies $-\left\langle \phi,\phi\right\rangle \le\left\langle j_{l},\cL^{-1}j_{l}\right\rangle $.
Note that, since $\cH_{1}$ is given by completion, we realize such
an element by a Cauchy sequence. The requirement \textbf{$\left\langle \div\phi_{n}-j_{l},g\right\rangle \xrightarrow{n\to\infty}0$}
for all local $g$ is weaker than the definition of $V_{l}$, i.e.
\textbf{$\left\langle \div\phi_{n}-j_{l},\div\phi_{n}-j_{l}\right\rangle \xrightarrow{n\to\infty}0$};
we prove here the more general version in case it might come handy.
Finally, the second claim shows that the supremum is attained for
some $\phi$, expressed as a Cauchy sequence $(\phi_{T})_{T\ge0}$.
\end{proof}
The proof of Theorem \ref{thm:thomson} follows directly: by equation
(\ref{eq:j_div_phi}), we can write $V_{l}$ as $\phi_{l}+V_{0}$,
so
\begin{multline*}
\left\langle j_{l},\cL^{-1}j_{l}\right\rangle =\sup_{\phi\in V_{l}}\left\{ -\left\langle \phi,\phi\right\rangle \right\} =\sup_{\psi\in V_{0}}\left\{ -\left\langle \phi_{l}-\psi,\phi_{l}-\psi\right\rangle \right\} \\
=-\left\langle \phi_{l},\phi_{l}\right\rangle +\sup_{\psi\in V_{0}}\left\{ 2\left\langle \phi_{l},\psi\right\rangle -\left\langle \psi,\psi\right\rangle \right\} .
\end{multline*}
By \cite[equation II.2.33]{Spohn2012IPS} together with equation (\ref{eq:phil_phil})
(see also Remark \ref{rem:dirichlet}), the Green-Kubo formula (\ref{eq:green-kubo})
can be written as
\[
l\cdot Dl=\frac{1}{\chi}\left(\left\langle \phi_{l},\phi_{l}\right\rangle +\left\langle j_{l},\cL^{-1}j_{l}\right\rangle \right),
\]
yielding equation (\ref{eq:thomson}).\qed

\section{\label{sec:BT}Three times Bertini-Toninelli}

The focus of this section is to apply Theorem \ref{thm:thomson} in
order to bound the diffusion coefficient of the Bertini-Toninelli
model. We will demonstrate three different techniques to find test
flows in $V_{0}$ that can be used in equation (\ref{eq:thomson})
and see what bounds they provide. At the end of the section, we will
see how to combine these flows together to further improve our estimate.

The one dimensional Bertini-Toninelli model \cite{BertiniToninelli}
is given by the transition rates 
\begin{equation}
c(x,x+1;\eta)=1-\eta(x-1)\eta(x+2),\label{eq:c_BT}
\end{equation}
that is, a transition from $\eta$ to $\eta^{x,x+1}$ occurs with
rate $1$ when $\eta(x-1)$ is empty or $\eta(x+2)$ is empty.

Since in dimension $1$ the diffusion coefficient is just a number,
we will fix $l=1$ (keeping the subscript $l$ for clarity).

Before bounding the diffusion coefficient from below, let us write
down for later reference the upper bound given by (\ref{eq:dirichlet})
with the test function $f\equiv0$:
\begin{equation}
D\le1-p^{2}=q(2-q).\label{eq:BT_upperbound}
\end{equation}

We start with two useful formulas, exploiting the fact that the Bertini-Toninelli
model is one dimensional with nearest neighbor jumps.
\begin{claim}
\label{claim:BT_formulas}Let $\phi:\Omega^{2}\to\R$ be an admissible
flow. Then 
\begin{eqnarray*}
\left\langle \phi,\phi\right\rangle  & = & \frac{1}{2}\mu\left[\left(\sum_{x\in\Z}\tau_{x}\phi(\eta,\eta^{0,1})\right)^{2}\right],\\
\left\langle \phi_{l},\phi\right\rangle  & = & \frac{1}{2}\sum_{y}\mu\left[\phi(\eta,\eta^{y,y+1})\,(\eta(y)-\eta(y+1))\right].
\end{eqnarray*}
\end{claim}

\begin{proof}
Direct calculation from the definition of $\phi_{l}$ and the inner
product.
\end{proof}

\subsection{Direct test with the flow of a gradient model}

The first test flow we construct is inspired by the observation \cite{GoncalvesLandimToninelli,Shapira2024Noncooperative}
that the Bertini-Toninelli model is dominated by a gradient model.
Indeed, one may define a model with transition rates 
\begin{equation}
c^{+}(x,x+1)=1-\frac{1}{2}\eta(x-1)-\frac{1}{2}\eta(x+2),
\end{equation}
that we will call the +-model. This model is on one hand gradient,
and on the other hand $c^{+}(x,x+1;\eta)\le c(x,x+1;\eta)$ for all
$x\in\Z,\eta\in\Omega$. 

The fact that the +-model is gradient means that we can calculate
its diffusion coefficient explicitly, yielding $D^{+}=q$. Then, by
the minimization principle (\ref{eq:dirichlet}), 
\begin{equation}
D\ge q.\label{eq:BT_bound0}
\end{equation}

Using Theorem \ref{thm:thomson} we can do better: since the +-model
is gradient, its ``$\phi_{l}$'' is in $V_{0}$ (see Remark \ref{rem:gradient}),
and could be used as a test flow. Let 
\begin{equation}
\phi(\eta,\eta^{0,z})=\frac{1}{2}z\ c^{+}(0,z;\eta)\left(\eta(0)-\eta(z)\right),\qquad\eta\in\Omega,z\in\pm1.
\end{equation}
Then 
\begin{eqnarray}
\div\phi(\eta) & = & \frac{1}{2}\ c^{+}(0,1;\eta)\left(\eta(0)-\eta(1)\right)-\frac{1}{2}\ c^{+}(0,-1;\eta)\left(\eta(-1)-\eta(0)\right)\label{eq:div_phi1}\\
 & = & -\frac{1}{4}\eta(-1)\eta(0)-\frac{1}{4}\eta(2)\eta(0)+\frac{1}{4}\eta(-1)\eta(1)+\frac{1}{4}\eta(2)\eta(1)\nonumber \\
 &  & +\frac{1}{4}\eta(-2)\eta(-1)+\frac{1}{4}\eta(1)\eta(-1)-\frac{1}{4}\eta(-2)\eta(0)-\frac{1}{4}\eta(1)\eta(0).\nonumber 
\end{eqnarray}
As expected, $\left\langle \div\phi,g\right\rangle =0$ for all $g$:
indeed, $\mu(\tau_{x}\phi)$ is independent of $\phi$ whenever $|x|>2$.
We may therefore replace the configuration $\eta\in\Omega$ with a
configuration on a large torus, e.g. $\Omega_{100}=\{0,1\}^{\Z/100\Z}$,
so (noting $\mu(\div\phi)=0$)
\[
\left\langle \div\phi,g\right\rangle =\mu\left[\left(\sum_{x\in\Z/100\Z}\tau_{x}\div\phi(\eta)\right)\cdot g\right].
\]
Let us now look closely at the sum $\sum_{x\in\Z/100\Z}\tau_{x}\div\phi(\eta)$:
each pair of neighboring particles contributes $-\frac{1}{4}+\frac{1}{4}+\frac{1}{4}-\frac{1}{4}=0$,
corresponding to the terms in the sum (\ref{eq:div_phi1}) with neighboring
sites; and each pair of particles at distance $2$ contributes $-\frac{1}{4}+\frac{1}{4}+\frac{1}{4}-\frac{1}{4}=0$,
corresponding to the terms in the sum (\ref{eq:div_phi1}) with sites
at distance $2$. Overall, $\sum_{x\in\Z/100\Z}\tau_{x}\div\phi(\eta)=0$.

Since $c^{+}(x,y;\eta)\le c(x,y;\eta)$, the flow $\phi$ is admissible,
and we conclude that it is in $V_{0}$. We may therefore use this
test flow in equation (\ref{eq:thomson}).

By direct calculation (using Claim \ref{claim:BT_formulas}): 
\begin{eqnarray}
\left\langle \phi,\phi\right\rangle  & = & \frac{1}{2}pq^{2}(1+q),\\
\left\langle \phi_{l},\phi\right\rangle  & = & pq^{2}.
\end{eqnarray}
Then, by Corollary \ref{cor:optimize_flow},
\begin{equation}
D\ge\frac{1}{\chi}\ \frac{\left\langle \phi_{l},\phi\right\rangle ^{2}}{\left\langle \phi,\phi\right\rangle }=\frac{1}{pq}\frac{p^{2}q^{4}}{\frac{1}{2}pq^{2}(1+q)}=\frac{2q}{1+q}.
\end{equation}

\begin{rem}
\label{rem:BT_qualitative}The bound we obtain is better than the
immediate, previously known bound (\ref{eq:BT_bound0}). Moreover,
in the dense regime $q\ll1$ (which is the physically interesting
one), this bound matches the upper bound (\ref{eq:BT_upperbound})
up to a factor $1+O(q)$ (compared to the factor $2$ of (\ref{eq:BT_bound0})).
Qualitatively, it means that the Bertini-Toninelli model behaves,
up to a small correction, in the same way as the +-model (accelerated
by $2$). This shouldn't come as a surprise: for small $q$, even
though some transitions in the +-model occur with rate $1$, most
occur with rate $\frac{1}{2}$; hence it is reasonable to neglect
the rare ``fast'' transitions, obtaining the Bertini-Toninelli model
(decelerated by a factor $2$). Using Theorem \ref{thm:thomson} allows
us to exploit this heuristics, since rare transitions count less when
taking \emph{expectation} of $\frac{c^{+}(x,y;\eta)}{c(x,y;\eta)}$.
This is in contrary to the bound (\ref{eq:BT_bound0}) which requires
a factor \emph{uniformly} bounding $\frac{c^{+}(x,y;\eta)}{c(x,y;\eta)}$.
\end{rem}

\subsection{A path to a gradient model}

The situation in the last paragraph is rather lucky: general models
might not have a gradient model with transition rates bounded by $c$.
One remedy is to compare rates of a single exchange not directly,
but via a short path of allowed transitions. \cite{Shapira2024Noncooperative}
discusses a large family of models that can be compared via paths
to gradient models. To simplify the exposition, we stick to the Bertini-Toninelli
model, replacing the gradient model of the previous section with another
(less natural) one.

\cite[Section 6.1]{Shapira2024Noncooperative} construct a family
of gradient models, and we will consider here the one given by the
rates:
\begin{multline}
c^{2}(x,x+1;\eta)=\overline{\eta}(x-3)\overline{\eta}(x-2)+\overline{\eta}(x-1)\overline{\eta}(x+3)+\overline{\eta}(x+2)\overline{\eta}(x+4),
\end{multline}
using the notation 
\[
\overline{\eta}(x)=1-\eta(x).
\]
We can see that the second and third term correspond to transitions
where $c(x,x+1;\eta)=1$ (since $\eta(x-1)=0$ or $\eta(x+2)=0$).
The first term, however, can be nonzero even if $c(x,x+1;\eta)=0$.
Nonetheless, when $\overline{\eta}(x-3)\overline{\eta}(x-2)\neq0$,
we can still exchange $x$ with $x+1$ using three steps, see Figure
\ref{fig:finite_path}.
\begin{figure}
\begin{tikzpicture}[scale=0.3, every node/.style={scale=0.6}]
	\draw[step=1,gray] (0,0) grid +(7,1);
	\node[star, star point ratio=0.4, fill, scale=0.8] at (0.5,0.5) {};	
	\node[star, star point ratio=0.4, fill, scale=0.8] at (3.5,0.5) {};	
	\node[regular polygon,regular polygon sides=3, scale=0.4, fill] at (4.5,0.4) {};	
	\node[regular polygon,regular polygon sides=3, rotate=180, scale=0.4, fill] at (5.5,0.5) {};	
	\node[star, star point ratio=0.4, fill, scale=0.8] at (6.5,0.5) {};	
	
	\draw[->,shift={(0.5,0.5)}]  (7,0) to (8,0);
	
	\draw[step=1,gray] (9,0) grid +(7,1);
	\node[star, star point ratio=0.4, fill, scale=0.8] at (9.5,0.5) {};	
	\node[star, star point ratio=0.4, fill, scale=0.8] at (11.5,0.5) {};	
	\node[regular polygon,regular polygon sides=3, scale=0.4, fill] at (13.5,0.4) {};	
	\node[regular polygon,regular polygon sides=3, rotate=180, scale=0.4, fill] at (14.5,0.5) {};	
	\node[star, star point ratio=0.4, fill, scale=0.8] at (15.5,0.5) {};

	\draw[->,shift={(0.5,0.5)}]  (16,0) to (17,0);
	
	\draw[step=1,gray] (18,0) grid +(7,1);
	\node[star, star point ratio=0.4, fill, scale=0.8] at (18.5,0.5) {};	
	\node[star, star point ratio=0.4, fill, scale=0.8] at (20.5,0.5) {};	
	\node[regular polygon,regular polygon sides=3, rotate=180, scale=0.4, fill] at (22.5,0.5) {};	
	\node[regular polygon,regular polygon sides=3, scale=0.4, fill] at (23.5,0.4) {};	
	\node[star, star point ratio=0.4, fill, scale=0.8] at (24.5,0.5) {};	
	
	\draw[->,shift={(0.5,0.5)}]  (25,0) to (26,0);
	
	\draw[step=1,gray] (27,0) grid +(7,1);
	\node[star, star point ratio=0.4, fill, scale=0.8] at (27.5,0.5) {};	
	\node[star, star point ratio=0.4, fill, scale=0.8] at (30.5,0.5) {};	
	\node[regular polygon,regular polygon sides=3, rotate=180, scale=0.4, fill] at (31.5,0.5) {};	
	\node[regular polygon,regular polygon sides=3, scale=0.4, fill] at (32.5,0.4) {};	
	\node[star, star point ratio=0.4, fill, scale=0.8] at (33.5,0.5) {};	
	
\end{tikzpicture}

\caption{\label{fig:finite_path}The path described in equation (\ref{eq:finite_path}).
The sites $-3$ and $-2$ are initially empty, other sites may be
filled or empty, denoted by $\star$. The particle/hole at $0$ and
$1$ are marked as black triangles.}

\end{figure}
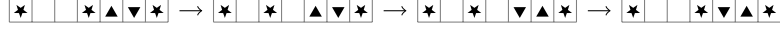

We will define a family of flows, depending on $\omega\in\Omega$
such that $c^{2}(0,1;\omega)\neq0$:
\begin{enumerate}
\item For $\omega\in\Omega$ such that $\overline{\omega}(-3)\overline{\omega}(-2)\neq0$,
let (see Figure \ref{fig:finite_path})
\begin{equation}
\omega_{1}=\omega,\ \omega_{2}=\omega_{1}^{-2,-1},\ \omega_{3}=\omega_{2}^{0,1},\ \omega_{4}=\omega_{3}^{-2,-1}=\omega^{0,1},\label{eq:finite_path}
\end{equation}
and the flow 
\begin{eqnarray*}
\psi_{\omega}^{1}(\eta,\eta') & = & \begin{cases}
\omega(0)-\omega(1) & \text{if }\exists i\in\{1,2,3\}\text{ such that }\text{\ensuremath{\eta=\omega_{i},\eta'=\omega_{i+1}},}\\
\omega(1)-\omega(0) & \text{if }\exists i\in\{1,2,3\}\text{ such that }\text{\ensuremath{\eta=\omega_{i+1},\eta'=\omega_{i}}},\\
0 & \text{otherwise.}
\end{cases}
\end{eqnarray*}
It is a direct verification that $\psi_{\omega}^{1}$ is antisymmetric
and admissible (thanks to the condition $\omega(-3)=\omega(-2)=0$).
We can also calculate 
\[
\div\psi_{\omega}^{1}(\eta)=\begin{cases}
\eta(0)-\eta(1) & \text{if }\eta=\omega\text{ or }\eta=\omega^{0,1},\\
0 & \text{otherwise.}
\end{cases}
\]
\item For $\omega\in\Omega$ such that $\overline{\omega}(-1)\overline{\omega}(3)\neq0$,
we simply take the flow 
\begin{equation}
\psi_{\omega}^{2}(\eta,\eta^{0,1})=\begin{cases}
\eta(0)-\eta(1) & \text{if }\eta=\omega\text{ or }\eta=\omega^{0,1},\\
0 & \text{otherwise}.
\end{cases}
\end{equation}
This is also an admissible flow, and 
\[
\div\psi_{\omega}^{2}(\eta)=\begin{cases}
\eta(0)-\eta(1) & \text{if }\eta=\omega\text{ or }\eta=\omega^{0,1},\\
0 & \text{otherwise.}
\end{cases}
\]
\item For $\omega\in\Omega$ such that $\overline{\omega}(2)\overline{\omega}(4)\neq0$,,
we take the flow 
\begin{equation}
\psi_{\omega}^{3}(\eta,\eta^{0,1})=\begin{cases}
\eta(0)-\eta(1) & \text{if }\eta=\omega\text{ or }\eta=\omega^{0,1},\\
0 & \text{otherwise}.
\end{cases}
\end{equation}
This is again an admissible flow, and 
\[
\div\psi_{\omega}^{3}(\eta)=\begin{cases}
\eta(0)-\eta(1) & \text{if }\eta=\omega\text{ or }\eta=\omega^{0,1},\\
0 & \text{otherwise.}
\end{cases}
\]
\end{enumerate}
We now define 
\begin{eqnarray*}
\psi^{1} & = & \sum_{\omega\in\Omega}\overline{\omega}(-3)\overline{\omega}(-2)\ \psi_{\omega}^{1},\\
\psi^{2} & = & \sum_{\omega\in\Omega}\overline{\omega}(-1)\overline{\omega}(3)\ \psi_{\omega}^{2},\\
\psi^{3} & = & \sum_{\omega\in\Omega}\overline{\omega}(2)\overline{\omega}(4)\ \psi_{\omega}^{3},\\
\psi & = & \psi^{1}+\psi^{2}+\psi^{3}.
\end{eqnarray*}
Then 
\begin{eqnarray*}
\div\psi^{1}(\eta) & = & \sum_{\omega}\overline{\omega}(-3)\overline{\omega}(-2)\ (\eta(0)-\eta(1))\One_{\eta=\omega\text{ or }\eta=\omega^{0,1}}\\
 & = & 2\ \overline{\eta}(-3)\overline{\eta}(-2)\ (\eta(0)-\eta(1)),\\
\div\psi^{2}(\eta) & = & 2\ \overline{\eta}(-1)\overline{\eta}(3)\ (\eta(0)-\eta(1)),\\
\div\psi^{3}(\eta) & = & 2\ \overline{\eta}(2)\overline{\eta}(4)\ (\eta(0)-\eta(1)),\\
\div\psi(\eta) & = & 2c^{2}(0,1;\eta)\ (\eta(0)-\eta(1)).
\end{eqnarray*}
By the same argument in the previous section (after equation (\ref{eq:div_phi1})),
we obtain $\psi\in V_{0}$.

A somewhat technical but straightforward calculation yields:
\begin{eqnarray*}
\sum_{x}\tau_{x}\psi^{1}(\eta,\eta^{0,1}) & = & 2(\overline{\eta}(3)-\overline{\eta}(2)+\overline{\eta}(-3))\ \overline{\eta}(-1))\ (\eta(0)-\eta(1)),\\
\sum_{x}\tau_{x}\psi^{2}(\eta,\eta^{0,1}) & = & 2(\overline{\eta}(-1)\overline{\eta}(3)\ (\eta(0)-\eta(1)),\\
\sum_{x}\tau_{x}\psi^{3}(\eta,\eta^{0,1}) & = & 2\overline{\eta}(2)\overline{\eta}(4)\ (\eta(0)-\eta(1)),\\
\sum_{x}\tau_{x}\psi(\eta,\eta^{0,1}) & = & 2\left(2\overline{\eta}(-1)\overline{\eta}(3)-\overline{\eta}(-1)\overline{\eta}(2)+\overline{\eta}(-3)\overline{\eta}(-1)+\overline{\eta}(2)\overline{\eta}(4)\right)(\eta(0)-\eta(1)).
\end{eqnarray*}
Note that, not surprisingly, when this expression is nonzero $c(0,1;\eta)=1$,
so we may calculate
\[
\left\langle \psi,\psi\right\rangle =\frac{1}{2}\mu\left[\left(\sum_{x\in\Z}\tau_{x}\psi(\eta,\eta^{0,1})\right)^{2}\right]=4pq^{3}\left(7-4q+6q^{2}\right)
\]
and
\[
\left\langle \phi_{l},\psi\right\rangle =6pq^{3}.
\]
We conclude using Corollary \ref{cor:optimize_flow}:
\[
D\ge\frac{1}{\chi}\ \frac{\left\langle \phi_{l},\psi\right\rangle ^{2}}{\left\langle \psi,\psi\right\rangle }=\frac{1}{pq}\frac{36p^{2}q^{6}}{4pq^{3}\left(7-4q+6q^{2}\right)}=\frac{9q^{2}}{7-4q+6q^{2}}.
\]

\subsection{\label{subsec:unbounded_path}Unbounded path to the simple exclusion
process}

The last test flow we construct will also be via a path, but rather
than a carefully chosen gradient model, we compare the Bertini-Toninelli
model to the simple exclusion process, where two neighboring sites
exchange occupation with rate $1$. Unlike the test flow $\psi$,
that required paths of length 1 or 3, we will need here paths of arbitrary
length. This would become important when studying other models: as
discussed in \cite{Shapira23KA_HL,Shapira2024Noncooperative}, kinetically
constrained lattice gases can be divided in two classes determining
their qualitative behavior: \emph{cooperative} and \emph{non-cooperative}.
It is shown there that models that can be compared to a gradient model
via fixed size paths are all non-cooperative; hence cooperative models
can only be studied if we are able to use paths of unbounded length.

For a configuration $\omega\in\Omega$, define:
\begin{eqnarray*}
X(\eta) & = & \min\{x\ge0:\eta(x)=\eta(x+1)=0\},\\
X_{4}(\eta) & = & \min\{x\ge4:\eta(x)=\eta(x+1)=0\},\\
X_{+}(\eta) & = & \min\{y\ge X(\eta)+2:\eta(y)=\eta(y+1)=0\},\\
X_{-}(\eta) & = & \max\left\{ x\le-2:\eta(x)=\eta(x+1)=0\right\} .
\end{eqnarray*}
Following \cite{BertiniToninelli,Shapira2024Noncooperative} we will
construct a path, starting with some $\omega\in\Omega$, that will
use only transitions with positive rates, and end at $\omega^{0,1}$.
The following lemma gathers the properties of this path that we will
need when analyzing the corresponding flow (see also Figure \ref{fig:unbounded_path}).
\begin{lem}
\label{lem:path}For almost any $\omega\in\Omega$, $\omega(0)\neq\omega(1)$,
there exists a path $\omega_{1},\dots,\omega_{N}$, for $N=N(\omega)\in2\N$,
and a sequence of sites $x_{1},\dots,x_{N-1}$ in $[0,X(\omega)]$,
such that $\omega_{N}=\omega^{0,1}$, $\omega_{i+1}=\omega_{i}^{x_{i},x_{i}+1}\neq\omega_{i}$,
and $c(x_{i},x_{i+1};\omega_{i})=1$ for any $i\in\{1,\dots,N-1\}$.
Moreover, the following properties are satisfied:
\begin{enumerate}
\item $x_{1},\dots,x_{N/2-1}$ are all in $[2,X(\omega)],$ $x_{N/2}=0$,
and $x_{i}=x_{N-i}$ for any $i\ge N/2+1$.
\item $\omega_{j}(0)-\omega_{j}(1)=\omega(0)-\omega(1)$ for $j\le N/2$
and $\omega_{j}(0)-\omega_{j}(1)=\omega(1)-\omega(0)$ for $j>N/2$.
\item For all $i\in\{1,\dots,N-1\}$, $\omega_{i}(x)=\omega(x)$ if $x\notin[0,X(\omega)+1]$.
\item For all $i\in\{1,\dots,N-1\}$, $X_{-}(\omega)=X_{-}(\omega_{i})$.
\item For all $i\in\{1,\dots,N-1\}$, $X_{+}(\omega)\le\min\{x\ge x_{i}+4:\omega_{i}(x)=\omega_{i}(x+1)=0\}$.
\item Let $\eta\in\Omega$, and assume that there exists $i\in\{1,\dots,N-1\}$
such that $\tau_{x_{i}}\eta=\omega_{i}$. Then 
\begin{enumerate}
\item $X_{-}(\omega)=X_{-}(\eta)+x_{i}$;
\item $X_{+}(\omega)\le X_{4}(\eta)+x_{i}\le X_{4}(\eta)-X_{-}(\eta)$.
\end{enumerate}
\item Let $\eta\in\Omega$, $x,x'\in\Z$, $\zeta=\pm1$. Then there is at
most one possible configuration $\omega$ and one possible $i\in\{1,\dots,N(\omega)\}$
such that
\begin{enumerate}
\item $x=x_{i}$ and $\eta=\eta_{i}$,
\item $X(\omega)=x'$, and
\item $\omega(0)-\omega(1)=\zeta$.
\end{enumerate}
\item For $\omega'=\omega^{0,1}$, the path $\omega'_{1},\dots,\omega_{N}'$
beginning with $\omega'$ is given by the reversed path $\omega'_{k}=\omega_{N-k}$;
and the associated sites $x_{1}',\dots,x_{N-1}'$ are the same as
$x_{N-1},\dots,x_{1}$, which are also the same as $x_{1},\dots,x_{N-1}$.
\end{enumerate}
\end{lem}

\begin{proof}
The construction of the path is based on the same idea as \cite{BertiniToninelli,Shapira2024Noncooperative},
see Figure \ref{fig:unbounded_path}. If $c(0,1;\omega)=1$, we can
simply exchange $0$ and $1$, so we may assume in the following that
$X(\omega)\ge3$. Start with the pair of empty sites $X(\omega),X(\omega+1)$.
Then we may exchange $X(\omega)$ with $X(\omega)-1$ and then $X(\omega)+1$
with $X(\omega)$, reaching a state $\omega_{3}$ where the empty
pair moved one step to the left; so $X(\omega_{3})=X(\omega)-1$.
Continue iteratively, until reaching $\omega_{k}$ such that $X(\omega_{k})=2$.
Then exchange sites $0$ and $1$, and roll back the previous transitions
to bring the empty pair at $2$ and $3$ back to $X(\omega),X(\omega)+1$.
The properties listed above can be verified one by one.
\end{proof}
\begin{figure}
\begin{tikzpicture}[scale=0.3, every node/.style={scale=0.6}]
	\def \x{0};
	\def \y{0};
	\draw[step=1,gray,shift={(\x,\y)}] (0,0) grid +(7,1);
	
	\draw[shift={(\x,\y)}] (0.5,0.4) node[regular polygon,regular polygon sides=3, scale=0.4, fill] {};	
	\draw[shift={(\x,\y)}] (1.5,0.5) node[regular polygon,regular polygon sides=3, scale=0.4, fill, rotate=180] {};
	\draw[shift={(\x,\y)}] (2.5,0.5) node[circle,fill, scale=0.8] {};
	\draw[shift={(\x,\y)}] (3.5,0.5) node[circle,fill, scale=0.8] {};
	\draw[shift={(\x,\y)}] (6.5,0.5) node[circle,fill, scale=0.8] {};
	
	\draw[->,shift={(\x + 0.5,\y + 0.5)}]  (7,0) to (8,0);
	
	\def \x{9};
	\def \y{0};
	\draw[step=1,gray,shift={(\x,\y)}] (0,0) grid +(7,1);

	\draw[shift={(\x,\y)}] (0.5,0.4) node[regular polygon,regular polygon sides=3, scale=0.4, fill] {};	
	\draw[shift={(\x,\y)}] (1.5,0.5) node[regular polygon,regular polygon sides=3, scale=0.4, fill, rotate=180] {};
	\draw[shift={(\x,\y)}] (2.5,0.5) node[circle,fill, scale=0.8] {};
	\draw[shift={(\x,\y)}] (4.5,0.5) node[circle,fill, scale=0.8] {};
	\draw[shift={(\x,\y)}] (6.5,0.5) node[circle,fill, scale=0.8] {};
	
	\draw[->,shift={(\x + 0.5,\y + 0.5)}]  (7,0) to (8,0);
	
	\def \x{18};
	\def \y{0};
	\draw[step=1,gray,shift={(\x,\y)}] (0,0) grid +(7,1);

	\draw[shift={(\x,\y)}] (0.5,0.4) node[regular polygon,regular polygon sides=3, scale=0.4, fill] {};	
	\draw[shift={(\x,\y)}] (1.5,0.5) node[regular polygon,regular polygon sides=3, scale=0.4, fill, rotate=180] {};
	\draw[shift={(\x,\y)}] (2.5,0.5) node[circle,fill, scale=0.8] {};
	\draw[shift={(\x,\y)}] (5.5,0.5) node[circle,fill, scale=0.8] {};
	\draw[shift={(\x,\y)}] (6.5,0.5) node[circle,fill, scale=0.8] {};
	
	\draw[->,shift={(\x + 0.5,\y + 0.5)}]  (7,0) to (8,0);
	
	\def \x{27};
	\def \y{0};
	\draw[step=1,gray,shift={(\x,\y)}] (0,0) grid +(7,1);

	\draw[shift={(\x,\y)}] (0.5,0.4) node[regular polygon,regular polygon sides=3, scale=0.4, fill] {};	
	\draw[shift={(\x,\y)}] (1.5,0.5) node[regular polygon,regular polygon sides=3, scale=0.4, fill, rotate=180] {};
	\draw[shift={(\x,\y)}] (3.5,0.5) node[circle,fill, scale=0.8] {};
	\draw[shift={(\x,\y)}] (5.5,0.5) node[circle,fill, scale=0.8] {};
	\draw[shift={(\x,\y)}] (6.5,0.5) node[circle,fill, scale=0.8] {};
	
	\draw[->,shift={(\x + 0.5,\y + 0.5)}]  (7,0) to (8,0);
	
	\def \x{36};
	\def \y{0};
	\draw[step=1,gray,shift={(\x,\y)}] (0,0) grid +(7,1);

	\draw[shift={(\x,\y)}] (0.5,0.4) node[regular polygon,regular polygon sides=3, scale=0.4, fill] {};	
	\draw[shift={(\x,\y)}] (1.5,0.5) node[regular polygon,regular polygon sides=3, scale=0.4, fill, rotate=180] {};
	\draw[shift={(\x,\y)}] (4.5,0.5) node[circle,fill, scale=0.8] {};
	\draw[shift={(\x,\y)}] (5.5,0.5) node[circle,fill, scale=0.8] {};
	\draw[shift={(\x,\y)}] (6.5,0.5) node[circle,fill, scale=0.8] {};
	
	\draw[->,shift={(0,-1.5)}]  (0,0) to (1,0);
	
	\def \x{2};
	\def \y{-2};
	\draw[step=1,gray,shift={(\x,\y)}] (0,0) grid +(7,1);

	\draw[shift={(\x,\y)}] (0.5,0.5) node[regular polygon,regular polygon sides=3, scale=0.4, fill, rotate=180] {};	
	\draw[shift={(\x,\y)}] (1.5,0.4) node[regular polygon,regular polygon sides=3, scale=0.4, fill] {};
	\draw[shift={(\x,\y)}] (4.5,0.5) node[circle,fill, scale=0.8] {};
	\draw[shift={(\x,\y)}] (5.5,0.5) node[circle,fill, scale=0.8] {};
	\draw[shift={(\x,\y)}] (6.5,0.5) node[circle,fill, scale=0.8] {};
	
	\draw[->,shift={(\x + 0.5,\y + 0.5)}]  (7,0) to (8,0);
	
	\def \x{11};
	\def \y{-2};
	\draw[step=1,gray,shift={(\x,\y)}] (0,0) grid +(7,1);

	\draw[shift={(\x,\y)}] (0.5,0.5) node[regular polygon,regular polygon sides=3, scale=0.4, fill, rotate=180] {};	
	\draw[shift={(\x,\y)}] (1.5,0.4) node[regular polygon,regular polygon sides=3, scale=0.4, fill] {};
	\draw[shift={(\x,\y)}] (3.5,0.5) node[circle,fill, scale=0.8] {};
	\draw[shift={(\x,\y)}] (5.5,0.5) node[circle,fill, scale=0.8] {};
	\draw[shift={(\x,\y)}] (6.5,0.5) node[circle,fill, scale=0.8] {};
	
	\draw[->,shift={(\x + 0.5,\y + 0.5)}]  (7,0) to (8,0);
	
	\def \x{20};
	\def \y{-2};
	\draw[step=1,gray,shift={(\x,\y)}] (0,0) grid +(7,1);

	\draw[shift={(\x,\y)}] (0.5,0.5) node[regular polygon,regular polygon sides=3, scale=0.4, fill, rotate=180] {};	
	\draw[shift={(\x,\y)}] (1.5,0.4) node[regular polygon,regular polygon sides=3, scale=0.4, fill] {};
	\draw[shift={(\x,\y)}] (2.5,0.5) node[circle,fill, scale=0.8] {};
	\draw[shift={(\x,\y)}] (5.5,0.5) node[circle,fill, scale=0.8] {};
	\draw[shift={(\x,\y)}] (6.5,0.5) node[circle,fill, scale=0.8] {};
	
	\draw[->,shift={(\x + 0.5,\y + 0.5)}]  (7,0) to (8,0);
	
	\def \x{29};
	\def \y{-2};
	\draw[step=1,gray,shift={(\x,\y)}] (0,0) grid +(7,1);

	\draw[shift={(\x,\y)}] (0.5,0.5) node[regular polygon,regular polygon sides=3, scale=0.4, fill, rotate=180] {};	
	\draw[shift={(\x,\y)}] (1.5,0.4) node[regular polygon,regular polygon sides=3, scale=0.4, fill] {};
	\draw[shift={(\x,\y)}] (2.5,0.5) node[circle,fill, scale=0.8] {};
	\draw[shift={(\x,\y)}] (4.5,0.5) node[circle,fill, scale=0.8] {};
	\draw[shift={(\x,\y)}] (6.5,0.5) node[circle,fill, scale=0.8] {};
	
	\draw[->,shift={(\x + 0.5,\y + 0.5)}]  (7,0) to (8,0);
	
	\def \x{38};
	\def \y{-2};
	\draw[step=1,gray,shift={(\x,\y)}] (0,0) grid +(7,1);

	\draw[shift={(\x,\y)}] (0.5,0.5) node[regular polygon,regular polygon sides=3, scale=0.4, fill, rotate=180] {};	
	\draw[shift={(\x,\y)}] (1.5,0.4) node[regular polygon,regular polygon sides=3, scale=0.4, fill] {};
	\draw[shift={(\x,\y)}] (2.5,0.5) node[circle,fill, scale=0.8] {};
	\draw[shift={(\x,\y)}] (3.5,0.5) node[circle,fill, scale=0.8] {};
	\draw[shift={(\x,\y)}] (6.5,0.5) node[circle,fill, scale=0.8] {};
	
\end{tikzpicture}

\caption{\label{fig:unbounded_path}The path constructed in Lemma \ref{lem:path}.
The leftmost site is $0$, the initial occupation of $0$ and $1$
is marked by black triangles facing up and down. In this example $X(\omega)=4$.}

\end{figure}
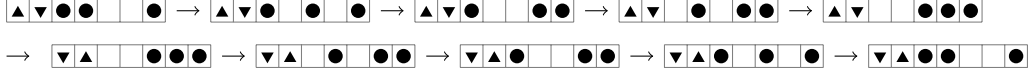
We can now use the path to define a flow: for any $\omega\in\Omega$,
with the notation of the lemma, let $\theta_{\omega}$ be the flow:
\[
\theta_{\omega}(\eta,\eta^{x,x+1})=\begin{cases}
\omega(0)-\omega(1) & \text{if }\exists i\in\{1,\dots,N-1\}\text{ such that }\eta=\omega_{i},x=x_{i},\\
\omega(1)-\omega(0) & \text{if }\exists i\in\{1,\dots,N-1\}\text{ such that }\eta=\omega_{i+1},x=x_{i},\\
0 & \text{otherwise},
\end{cases}
\]
and 
\[
\theta=\sum_{\omega}\theta_{\omega}.
\]
Then 
\begin{eqnarray*}
\div\theta_{\omega}(\eta) & = & \begin{cases}
\eta(0)-\eta(1) & \text{if }\eta=\omega\text{ or }\eta=\omega^{0,1},\\
0 & \text{otherwise};
\end{cases}\\
\div\theta(\eta) & = & 2\left(\eta(0)-\eta(1)\right).
\end{eqnarray*}
Hence, $\theta\in V_{0}$, and we are left with calculating $\left\langle \theta,\theta\right\rangle $
and $\left\langle \phi_{l},\theta\right\rangle $.

We can now use Lemma \ref{lem:path} to write
\begin{multline*}
\sum_{x\in\Z}\tau_{x}\theta(\eta,\eta^{0,1})=\sum_{x\in\Z}\sum_{\omega\in\Omega}\theta_{\omega}(\tau_{x}\eta,(\tau_{x}\eta)^{x,x+1})\\
=\sum_{x\in\Z}\sum_{\omega\in\Omega}\sum_{i=1}^{N(\omega)}\left(\omega(0)-\omega(1)\right)\One_{\tau_{x}\eta=\omega_{i}}\One_{x=x_{i}}-\sum_{x}\sum_{\omega}\sum_{i}\left(\omega(0)-\omega(1)\right)\One_{\tau_{x}\eta=\omega_{i+1}}\One_{x=x_{i}}\\
=2\sum_{\zeta=\pm1}\sum_{x'=0}^{X_{4}(\eta)-X_{-}(\eta)}\sum_{x=0}^{x'+1}\sum_{\omega\in\Omega}\sum_{i=1}^{N(\omega)}\ \zeta\ \One_{\tau_{x}\eta=\omega_{i}}\One_{x=x_{i}}\One_{\omega(0)-\omega(1)=\zeta}\One_{x'=X(\omega)}.
\end{multline*}
By point 6 of Lemma \ref{lem:path} all the indicators allow us to
remove the sums over $\omega$ and $i$, yielding
\[
\left|\sum_{x\in\Z}\tau_{x}\theta(\eta,\eta^{0,1})\right|\le4\left(X_{4}(\eta)-X_{-}(\eta)+1\right)^{2},
\]
and therefore

\begin{eqnarray*}
\left\langle \theta,\theta\right\rangle  & \le & 8\mu\left[\left(X_{4}(\eta)-X_{-}(\eta)+1\right)^{4}\right].
\end{eqnarray*}
This expectation is explicit ($-2-X_{-}$ and $X_{3}-4$ are IID geometric
random variables with parameter $q^{2}$), and is bounded by $2401q^{-8}$,
hence 
\[
\left\langle \theta,\theta\right\rangle \le20000\ q^{-8}.
\]

Finally,
\begin{eqnarray*}
\left\langle \phi_{l},\theta\right\rangle  & = & \frac{1}{2}\sum_{y}\mu\left[\theta(\eta,\eta^{y,y+1})\,(\eta(y)-\eta(y+1))\right]\\
 & = & \frac{1}{2}\mu\left[\theta(\eta,\eta^{0,1})(\eta(0)-\eta(1))\right]+\frac{1}{2}\sum_{y\ge2}\mu\left[\theta(\eta,\eta^{y,y+1})\,(\eta(y)-\eta(y+1))\right]\\
 & = & \mu\left[(\eta(0)-\eta(1))^{2}\right]+\frac{1}{2}\sum_{y\ge2}\sum_{\omega}\mu\left[\theta_{\omega}(\eta,\eta^{y,y+1})\,(\eta(y)-\eta(y+1))\right].
\end{eqnarray*}
The second term of this sum is $0$: fix $\eta\in\Omega,\omega\in\Omega,y\ge2$
such that $y=x_{i}$ and $\eta=\omega_{i}$ for some $i$, so $\theta_{\omega}(\eta,\eta^{y,y+1})=\omega(0)-\omega(1)$.
By the properties of the path, if we consider the path beginning with
$\omega'=\omega^{0,1}$, we get $\eta=\omega_{N/2-i}^{\prime}$ and
$y=x_{N/2-i}'$. Then $\theta_{\omega'}(\eta,\eta^{y,y+1})=\omega'(0)-\omega'(1)=-\theta_{\omega}(\eta,\eta^{y,y+1})$.
The same argument holds for the case $y=x_{i}$ and $\eta=\omega_{i+1}$,
so in total
\[
\sum_{\omega}\theta_{\omega}(\eta,\eta^{y,y+1})=0,\qquad\forall y\ge2.
\]
Therefore
\begin{eqnarray*}
\left\langle \phi_{l},\theta\right\rangle  & = & 2pq,
\end{eqnarray*}
ending up with
\begin{equation}
D\ge\frac{1}{pq}\frac{4p^{2}q^{2}}{20000\ q^{-8}}=\frac{1}{5000}\ pq^{9}.
\end{equation}

\bigskip{}

\begin{rem}
\label{rem:combined_flow}To conclude this section, we note that any
linear combination $\lambda_{1}\phi+\lambda_{2}\psi+\lambda_{3}\theta$
is in $V_{0}$. To avoid some computational complication, let us stick
to the first two flows, maximizing 
\begin{eqnarray*}
F(\lambda_{1},\lambda_{2}) & = & 2\left\langle \phi_{l},\lambda_{1}\phi+\lambda_{2}\psi\right\rangle -\left\langle \lambda_{1}\phi+\lambda_{2}\psi,\lambda_{1}\phi+\lambda_{2}\psi\right\rangle \\
 & = & 2\left\langle \phi_{l},\phi\right\rangle \lambda_{1}+2\left\langle \phi_{l},\psi\right\rangle \lambda_{2}-\left\langle \phi,\phi\right\rangle \lambda_{1}^{2}-\left\langle \psi,\psi\right\rangle \lambda_{2}^{2}-2\left\langle \phi,\psi\right\rangle \lambda_{1}\lambda_{2}
\end{eqnarray*}
in the case $q=p=1/2$. Then (after computations)
\[
F(\lambda_{1},\lambda_{2})=\frac{1}{4}\lambda_{1}+\frac{3}{4}\lambda_{2}-\frac{3}{32}\lambda_{1}^{2}-\frac{13}{8}\lambda_{2}^{2}-\frac{7}{16}\lambda_{1}\lambda_{2},
\]
whose maximum gives $\chi D\ge\frac{37}{214}\approx0.173$. This could
be compared with the bound given by $\phi$, $\chi D\ge\frac{1}{6}\approx0.167$
and the one given by $\psi$, $\chi D\ge\frac{9}{104}\approx0.087$.
\end{rem}

\bibliographystyle{amsalpha}
\bibliography{thomson_diffusion}

\end{document}